\documentclass[10pt,leqno]{amsart}
\usepackage{amsmath, amsthm, amscd, amsfonts, amssymb}
\usepackage[fleqn,tbtags]{mathtools}
\usepackage{graphicx}\usepackage{color}\usepackage[colorlinks]{hyperref}
\usepackage[all]{xy}
\usepackage[most]{tcolorbox}
\usepackage{tikz-opm}
\title{Radical-injectivy in the category \textbf{S-Act}}
\author[M. Haddadi, S.M. N.Sheykholislami]{M. Haddadi, S.M. N.Sheykholislami}

\newcommand{\Con}{{\mathtt{{Con}}}}
\newcommand{\Dom}{{\mathtt{{Dom}}}}
\address{}

\email{}

\address{$^{3}$ Faculty of Mathematics and Computer Sciences, Semnan University, Semnan, Iran.}

\email{m.haddadi@semnan.ac.ir, }

\date{}
\begin{document}
\newtheorem{theorem}{Theorem}[section]
\newtheorem{lemma}[theorem]{Lemma}
\newtheorem{proposition}[theorem]{Proposition}
\newtheorem{corollary}[theorem]{Corollary}
\theoremstyle{definition}
\newtheorem{definition}[theorem]{Definition}
\newtheorem{example}[theorem]{Example}
\newtheorem{exercise}[theorem]{Exercise}
\newtheorem{remark}[theorem]{Remark}
\newtheorem{note}[theorem]{Note}

\newtheorem{My Guess}[theorem]{My Guess}

\maketitle


\begin{abstract}
Various generalizations of the concept of injectivy, in particular injectivy with respect to a specific class of morphisms,  have been intensively studied throughout the years in different categories. One of the important kinds of injectivy  studied in the category  {\bf R-Mod} of $R$-modules is  $\tau$-injectivy, for a torsion theory $\tau$, or in the other words $r$-injectivy, where  $r$ is the induced idempotent radical by $\tau$.

 In this paper,  we introduce the notion of $r$-injectivy, for a Hoehnke radical $r$  in the category {\bf S-Act} of $S$-acts and we study the main properties of this kind of injectivy. Indeed,  we show that this kind of injectivy is well behavior and also we  present a Bear Theorem for $r$-injective $S$-acts. We then consider $r$-injectivy for a Kurosh-Amitsur radical $r$ and  we give stronger results in this case. Finally we present conditions under which $r$-injective $S$-acts are exactly injective ones and we give a characterization for injective $S$-acts.
\\

\\
: Radical, S-act, Injectivy, $r$-injectivy.

\end{abstract}

\section{Introduction and Priminaries}

Injectivy and its various generalizations,  important and interesting for their own and also tightly related to certain concepts such as purity and etc, have been intensively studied
throughout the years in different categories \cite{ban, Ebrahimi (2010), Ebrahimi (2009), Ebrahimi (2014), Maranda, mcmorris, Tholen(2008)}. One of the important  kinds of injectivy for module theorists is  $\tau$-injectivy, for a torsion theory $\tau$, or in the other words $r$-injectivy in which  $r$ is the induced idempotent radical by $\tau$, \cite{Bland, Crivei (2004), Dickson, Lambek, Nishida}.

 In this paper first, with every Hoehnke radical $r$ we associate  a closure operator $c^r$ and consider the class of all $c^r$-dense monomorphisms so-called $r$-monomorphisms. We then, in Section 3, study the properties of the class of $r$-monomorphisms. In Sections 4 and 5 we consider the injective $S$-acts relative to $r$-monomorphisms, $r$-injective $S$-acts, and we study the main properties of this kind of injectivy and we establish the well behavior theorems for $r$-injectivy. We then give Bear-Skornjakov criterion for $r$-injective $S$-acts and weakly injective $S$-acts in Section \ref{6}. Then, in section \ref{7}, we investigate $r$-injectivy when $r$ is a Kurosh-Amitsur radical and  we get stronger results in this case. Finally, the relationship between $r$-injectivy and usual injectivy is analyzed. Indeed, we present conditions under which $r$-injective $S$-acts are exactly injective ones and we give a characterization for the usual injective $S$-acts.



 Now Let us recall some necessary notions. An $S$-act over a monoid $S$ is a set $A$ together with an action $(s,a)\mapsto as$, for $a\in A$, $s\in S$, subject to the rules
$t(sa)=(ts)a$ and $1a=a$, where $1$ is the identity element of the monoid $S$, for  all $a\in A$ and $s,t\in S$.
A \textit{ homomorphism of $S$-acts} is a map $f:A\rightarrow B$ subject to $f(sa)=sf(a)$, for all $a\in A$ and $s\in S$.
We will work in the category of all $S$-acts and homomorphisms between them.
 An $S$-act $A$ is said to be  trivial, if $|A|\leq 1$.

\medskip
An equivalence relation $\rho$ on an $S$-act  $A$ is called a \textit{congruence} on $A$, if  $a\rho a'$ implies $ (sa)\rho (sa')$, for all $s\in S$. We denote the set  of all congruences on $A$ by $\Con(A)$ which forms a bounded lattice in which the diagonal relation $\Delta_{A}=\{ (a,a)\ | \ a\in A\}$ is the smallest element and the total relation $\nabla_{A}=\{ (a,b)\ | \ a,b \in A\}$ is the grates
one.
 Every congruence $\chi \in \Con(A)$ determines a partition  of $A$ into $\chi$-classes and a system $\Sigma_{\chi}$ of those $\chi$-classes each of which is a non-trivial subact of $A$. Of course, $\Sigma_{\chi} $ may be empty.  Throughout this paper  we use the general notion of Rees congruence, as well as \cite{Wiegandt (2006)} instead of the usual Rees congruence defined in \cite{Kilp (2000)}, meaning that a congruence $\rho$ is a \textit{Rees congruence } if the $\rho$-cosets either are subacts or consists of one element. So every system $\Sigma$ of disjoint non-trivial subacts  of an $S$-act $A$ determines a Rees congruence  $\rho_{\Sigma}$ given by
\[(a,b)\in \rho_{\Sigma} \Longleftrightarrow \left\{\begin{matrix}
a,b\in B\ & \text{for some } B\in \Sigma \\
a=b\  \  \ & \text{otherwise.\   \ \  \  \  \  \  }
\end{matrix}
\right. \]
 We call $\rho_{\Sigma}$ to be the \textit{generated Rees congruence by $\Sigma$} on  $A$ and  $A/\rho_\Sigma$ a {\em Rees factor} of $A$ over $\rho_{\Sigma}$. Also we use the notion $\rho_{B}$ instead  of $\rho_{\Sigma}$ when $\Sigma$ is the singleton set $\{B\}$ and denote the Rees factor of over $\rho_{B}$ by $A/B$ instead of $A/\rho_{\Sigma}$.

A congruence $\chi_B$ of a subact $B$ of an $S$-act $A$ may extend to a congruence of the $S$-act
$A$. There is always the smallest extension $\chi_A$ given by
\[
(a, b)\in\chi_A\Longleftrightarrow\begin{cases}
(a, b) \in\chi_B\\
a = b \ \text{otherwise}.
\end{cases}
\]
Therefore we may consider each congruence $\chi_B\in\Con(B)$ as a congruence of
$\Con(A)$ by identifying $\chi_B$ and $\chi_A$. In particular,  $\nabla_B$ can be considered as the generated Rees
congruence  by $B$, $\rho_B \in { \Con}(A)$ .

\medskip
Now we give some different types of radicals in  $S$-Act which is usually considered.

\medskip
$\bullet $ An assignment $r:A\rightsquigarrow r(A)$ assigning  each $S$-act $A$ to a congruence $r(A)\in \Con(A)$ is called a \textit{Hoehnke radical} or simply a \textit{radical} whenever,

 \sloppypar \rm{(i)} every homomorphism $f:A\rightarrow B$ induces the a homomorphism
$r(f):r(A)\rightarrow r(B)$; meaning that $(f(a),f(a'))\in r(B)$ if
$(a,a')~\in ~r(A)$, for every homomorphism $f:A\rightarrow B$.

(ii)\ ~~~ $r(A/r(A))=\Delta_{A/r(A)}$.

\medskip
 $\bullet $ A  radical $r$ is said to be \textit{hereditary}, if $r(B)=r(A)\wedge \Delta_{B}$, for all $B\leq A$ and all $S$-acts $A$.

\medskip
 $\bullet $ A  radical $r$ of $S$-acts is called a \textit{Kurosh-Amitsur radical}, if

\rm{(i)} r(A) is a Rees congruence, for all $S$-acts $A$,

\rm{(ii)} for every $B\in \Sigma_{r(A)}$,  $r(B)=\nabla_B$.

 \medskip
With every  radical $r$ one can associate two classes of $S$-acts, namely \textit{radical class} (or \textit{torsion class}) $\mathbb{R}_{r}= \{ A \ |\ r(A)=\nabla_{A}\}$ and \textit{semisimple class} (or \textit{torsion-free class})  $\mathbb{S}_{r}=\{  A\ |\ r(A)=\Delta_{A}\}$. We call the members of $\mathbb{R}_{r}$ to be the {\em radical $S$-acts} and the members of $\mathbb{S}_{r}$ to be the {\em semisimple $S$-acts}.
It is worth  noting that $\mathbb{S}_{r}$ is closed under taking subacts, product, isomorphic copies and contains all trivial $S$-acts. Also every subclass $\mathbb{S}$ of $S$-acts which is closed under taking subacts, product, isomorphic copies and contains all trivial $S$-acts, determines a  radical $r_{\mathbb{S}}$ defined by
$ r_{\mathbb{S}}(A)=\wedge (\chi\in \Con(A)\ | \ A/\chi\in \mathbb{S})$.
Moreover, $\mathbb{S}=\mathbb{S}_{r}$ if and only if $r=r_{\mathbb{S}}$, see \cite{Wiegandt (2006)}.

We recall, from \cite{Wiegandt (2006), Haddadi}, that a subclass
$\mathbb{S}$ of $S$-acts is a semisimple class of a radical $r$
if and only if
\begin{enumerate}
\item $\mathbb{S}$ contains of all trivial $S$-acts,
\item $\mathbb{S}$ is closed under isomorphic copies,
\item  $\mathbb{S}$ is closed under taking subacts,
\item  $\mathbb{S}$ is closed under  products,
\item  $\mathbb{S}$ is closed under  congruence extensions. That is, $A/\chi \in\mathbb{S}$ and $\Sigma_{\chi}\subseteq \mathbb{S}$ imply $A\in \mathbb{S}$, for every $A\in$\textbf{S-Act} and every congruence $\chi$ on $A$.
\end{enumerate}

Also a subclass $\mathbb{R}$ of $S$-acts is a radical class of
a radical $r$ if and only if
\begin{enumerate}
\item $\mathbb{R}$ contains all trivial $S$-acts,
\item  $\mathbb{R}$ is homomorphically closed,
\item  $\mathbb{R}$ has the inductive property; that is $\bigcup_{i\in I} A_{i}\in \mathbb{R}$  , for every ascending chain $\{A_{i}\}_{i\in I}\leq \mathbb{R}$,
\item  $\mathbb{R}$ is closed under Rees extensions. That is $A/\rho \in\mathbb{R}$ and $\Sigma_{\rho}\subseteq \mathbb{R}$ imply $A\in \mathbb{R}$, for every $A\in$\textbf{S-Act} and every congruence $\rho$ on $A$.
\end{enumerate}

It is worth noting the following remark concerning  $ \Sigma_{r(A)}$, for every $S$-act $A$, where $r$ is a  radical.

\begin{remark}\label{•}
{\rm (i) } For every subact $B$ of $A$ with $B\in \mathbb{R}_r$, there exists $X\in \Sigma_{r(A)}$ such that $B\leq X$.

{\rm (ii) } Each $r(A)$-class $X$ containing a subact $B$ of $A$ is itself  a subact of $A$, and so  $X\in \Sigma_{r(A)}$.
\end{remark}

Now we recall the following lemma from \cite{Ebrahimi} which is used in the sequel.

\begin{lemma}\label{maranda}
Let $r$ be a   radical and  $\chi\subseteq r(A)$ be a congruence on an $S$-act $A$. Then $r(A/\chi)=r(A)/\chi$.
\end{lemma}

In particular for a Kurosh-Amitsur  radical $r$ and a set $\Sigma$ of disjoint subacts of an $S$-act $A$ with $\Sigma\leq \Sigma_{r(A)}$, we have $r(A/\rho_{\Sigma})=r(A)/\rho_{\Sigma}$.

\medskip
Also we recall,  given a subclass of monomorphisms $\mathcal{M}$, an $\mathcal{M}$-morphism $m$ is called to be $\mathcal{M}$-essential if for every homomorphism $f:B\rightarrow C$, $fm\in \mathcal{M}$ implies  $f\in \mathcal{M}$.
In this paper we use the terminology of  B. Banaschewski \cite{ban, Ebrahimi (2009), Ebrahimi (2014)} and we say that injectivy relative to a class $\mathcal{M}$  is well behaviour  in the category $\textbf{S\text{-}Act}$ if the following propositions are stablished.

\begin{proposition}[First well behaviour Theorem \cite{ban}]
 The following conditions are equivalent, for an $S$-act $A$:\\
\emph{(i)} $A$ is  $\mathcal{M}$-injective.\\
\emph{(ii)} $A$ is an $\mathcal{M}$-absolute retract.\\
\emph{(iii)} $A$ has no proper $\mathcal{M}$-essential extension.
\end{proposition}

\begin{proposition}[Second well behaviour Theorem \cite{ban}]
Every $S$-act  $A$ has an $\mathcal{M}$-injective hull.
\end{proposition}

\begin{proposition}[Third well behaviour Theorem \cite{ban}]
The following conditions are equivalent, for an $\mathcal{M}$-morphism $m:A\rightarrow B$ :\\
\emph{(i)} $B$ is an $\mathcal{M}$-injective hull of $A$.\\
\emph{(ii)} $B$ is a maximal $\mathcal{M}$-essential extension of $A$.\\
\emph{(iii)} $B$ is a minimal $\mathcal{M}$-injective extension of $A$.
\end{proposition}

\medskip
 A family  $C = (C_{A})_{A\in\textbf{S-Act}}$,  with $C_{A} : Sub(A) \to Sub(A)$, assigning every subact $B\leq A$ to a subact $C_{A}(B)$ (or simply $C(A)$ when no confusion arises) is called a \textit{closure operator} on \textbf{S-Act} if it satisfies the following properties:

(c1)\ (Extension)\ $B\leq C(B)$,

(c2)\ (Monotonicity)\ $B_{1} \leq B_{2}\leq A$ implies $C(B_{1}) \leq C(B_{2})$,

(c3)\ (Continuity)\ $f (C_{A}(B)) \leq C_{C} ( f (B))$, for all homomorphisms $f : A \to C$.

\medskip
A closure $C$ is called \textit{weakly hereditary} if   $C_{A}(B) = C_{C_{A}(B)}(B)$, for every subact $B$ of every $S$-act $A$.

A closure operator $C$ is called \textit{hereditary} if it holds $C_{B}(D) = C_{A}(D)\cap B$, for every subact $D \leq B \leq A$.

\medskip
 The readers may consult \cite{Adamek, Burris, Kilp (2000)} for the general facts about category theory and universal algebra  used in this paper. Here we also follow the notations and terminologies used there.

\section{The induced closure operator  from a  radical}

Usually radicals are a rich supply for the closure operators, See \cite{Tholen (1995)}. Hence we introduce a closure operator $c^{r}$, associated with a  radical $r$ and we describe the interrelationship of these two notions.

\begin{definition}
For a given  radical $r$ of {\bf S-Act}, we define a closure operator $c^{r}  $ in the category of {\bf S-Act} by
$ c_{A}^{r}(B)=\pi^{-1}([B]_{r(A/B)})$
in which $\pi: A\to A/B$ is the canonical epimorphism.
\end{definition}

A subact $B$ of $S$-act $A$ is said to be \textit{$r$-closed } if  $c^{r}(B)=B$ and it is said to be \textit{$r$-dense} if $c^{r}(B)=A$.
A monomorphism $ m:B \rightarrow A $ is said \textit{$r$-monomorphism} if $m (B)$ is $r$-dense in $A$.

One can easily check that $c^{r}$ is an idempotent closure operator, that is  $c_{A}^{r}(c_{A}^{r}(B))=c_{A}^{r}(B)$, for every $S$-act $A$ and every subact $B$ of $A$. So the class of $r$-closed subacts of an $S$-act $A$ is of the form $\{c_{A}^{r}(B)\ | \ B\leq A\}$.

It is worth noting that the class of $r$-monomorphisms is closed under composition. Also we have the following Lemma.

\begin{lemma}\label{hom-dens}
Let $B$ be an $r$-dense subact  of an $S$-act $A$ and $\chi$ be a congruence on $B$. Then $B/\chi$ is $r$-dense in $A/(\chi\vee \Delta_{A})$.
\end{lemma}

\begin{proof}
The result can easily follow from the following equations.
\[A/(\chi\vee \Delta_{A})/B/\chi=A/B \quad \& \quad  r(A/B)=\nabla_{A/B}\]
\end{proof}

\begin{proposition}\label{closed2act}
Let $r$ be a  radical and $B$ be an $r$-closed subact of an $S$-act  $A$. Then, for every $X_{A}\in \Sigma_{r(A)}$ and $X_{B}\in \Sigma_{r(B)}$,  we have $X_{A}\cap X_{B}=\emptyset$ or $X_{B}\leq X_{A}\leq B$.
\end{proposition}

\begin{proof}
To prove, we assume   $X_{A}\cap X_{B}\neq\emptyset$ and we show $X_{B}\leq X_{A}\leq B$. To do so, we consider the canonical epimorphism $\pi:A\to A/B$ and we have $\pi(r(A))\subseteq  r(A/B)$. Hence  there exists $Y_{A/B}\in \Sigma_{r(A/B)}$  such that $\pi(X_{A})\subseteq Y_{A/B}$, since $X_{A}\in \Sigma_{r(A)}$. But  since $ X_{A}\cap X_{B}\leq X_{A}\cap B$ is non-empty, the homomorphic image of $B$ under $\pi:A\to A/B$ is a zero element of $Y_{A/B}$. So $Y_{A/B}\leq [B]_{r(A/B)}$. Therefore $X_{A}\subseteq c_{A}^{r}(B)=B$.

Now  by considering the canonical epimorphism $\pi:B\to B/X_{A}$, we have $\pi(r(B))\subseteq  r(B/X_{A})$.  Hence there exists $Y_{B/X_{A}}\in \Sigma_{r(B/X_{A})}$ with $\pi(X_{B}),\pi(X_{A})\subseteq Y_{B/X_{A}}$, since $X_{B}\in \Sigma_{r(B)}$ and $X_{A}\cap X_{B}\neq\emptyset$. Also $[X_{A}]_{r(B/X_{A})}$ is singleton since, by Lemma  \ref{maranda}, $r(B/X_{A})\leq r(A/X_{A})\wedge \nabla_{B/X_{A}}=r(A)/X_{A}$. Therefore $X_{B}\leq \pi^{-1}(Y_{B/X_{A}})=X_{A}$.
\end{proof}

One is tempted to assume that the radical class of a radical $r$ is closed under coproduct. But this is not true in general, see example  $3.1$  from  \cite{Haddadi}. we recall some equivalent conditions with closedness of $\mathbb{R}_{r}$ under coproducts, for a Kurosh-Amitsur radical $r$, from \cite{Haddadi}, and then using the mentioned closure operator we give another characterization for the closedness  of $\mathbb{R}_{r} $ under coproduct in Theorem \ref{copro-clos}.

\begin{theorem}\label{copro-Kurosh-R}
Given a Kurosh-Amitsur radical $r$,  the following statements are equivalent
\begin{enumerate}
\item The class $\mathbb{R}_{r}$ is closed under coproduct.
\item For every $A\in \text{\bf{S-Act}}$,  $| \Sigma_{r(A)}| \leq 1$ and  $\mathbb{R}_{r}$ contains a non-trivial $S$-act.
\item For the trivial $S$-act $\Theta$,  $\Theta\amalg\Theta\in \mathbb{R}_{r}$.
\item For $B\leq A$,  $\pi^{-1}([B]_{r(A/B)})$ contains  all zeros of $A$ where $\pi: A\rightarrow A/B$ is the canonical epimorphism.
\end{enumerate}
\end{theorem}

As a corollary of the above theorem, we have:

\begin{corollary}\label{copro. zero}
Given a Kurosh-Amitsur radical $r$, the class $\mathbb{R}_{r}$ is closed under coproducts if and only if every $S$-act  $A$ has a  subact such as  $A_{r}\in \mathbb{R}_{r}$  such that  $A_{r}$ contains  all zeros of $A$ and $r(A)=\rho_{A_{r}}$.
\end{corollary}

\begin{theorem}\label{copro-clos}
Let $r$ be a Kurosh-Amitsur radical. Then $\mathbb{R}_{r}$ is closed under coproducts if and only if  $r(A/B)=\Delta_{A/B}$, for every $S$-act $A$ and every $r$-closed subact   $B$ of $A$.
\end{theorem}
\begin{proof}
$(\Rightarrow)$ Let $B$ be an $r$-closed subact of an $S$-act $A$. Then, by Corollary \ref{copro. zero}, $A/B$ has a subact such as $X_{A/B}\in \mathbb{R}_{r}$ such that $X_{A/B} $ contains all the zeros of $A/B$ and $r(A/B)=\rho_{X_{A/B}}$. But since the image of $B$ under  the canonical homomorphism $\pi:A\to A/B$ is a zero element of $A/B$,  $[B]_{r(A/B)}= X_{A/B}$. Also since  $B$ is an $r$-closed subact of $A$, we have    $B=\pi^{-1}([B]_{r(A/B)})=\pi^{-1}( X_{A/B})$. Therefore $X_{A/B}=B/B$ and hence $r(A/B)=\Delta_{A/B}$.

\medskip
$(\Leftarrow)$ For the converse we show the second assertion of Theorem \ref{copro-Kurosh-R}. Let $A$ be a non semisimple $S$-act and $B\in \Sigma_{r(A)}$. Then, by Lemma \ref{maranda}, $r(A/B)=r(A)/B$ and hence $B=\pi^{-1}([B]_{r(A/B)})$, where $\pi:A\to A/B$ is the canonical homomorphism. That is, $B$ is an $r$-closed subact of $A$. Now $r(A/B)=\Delta_{A/B}$ follows from the hypothesis. This means that $r(A)=\rho_{B}$. So, for every $S$-act $A$, $|\Sigma_{r(A)}|\leq 1$.
\end{proof}

\begin{proposition}\label{close coproduct}
Let $r$ be a   radical whose semisimple class is closed under coproducts  and $B$ be a proper $r$-dense subact of an $S$-act $A\in \mathbb{S}_{r}$. Then there exists $x\in A\setminus B$ and $s\in S$ such that $sx\in B$.
\end{proposition}
\begin{proof}
To prove, we suppose $sx\notin B$, for every $s\in S$ and $x\in A\setminus B$, and we get a contradiction.    Indeed, if $sx\notin B$, for every $s$ and $x$, then $A\setminus B$ is a subact of $A$ and so, $A/B \cong (A\setminus B)\amalg\Theta$, where $\Theta$ is a singleton trivial $S$-act. But since $\mathbb{S}_{r}$ is closed under taking subacts and $A\in \mathbb{S}_{r}$, $A\setminus B\in \mathbb{S}_{r}$. So $A/B=(A\setminus B) \amalg \Theta\in \mathbb{S}_{r}$ follows from the closedness of $\mathbb{S}_{r}$ under coproduct. Therefore $A/B\in \mathbb{S}_{r}\cap \mathbb{R}_{r}$ and hence $A/B$ is a trivial $S$-act. So $A=B$ and this contradicts the hypothesis.
\end{proof}

For some especial kind of radical more relations between $r$ and $c^{r}$ will display. In the following, we give some of them.

\begin{definition}
A  radical $r$ is called
\begin{enumerate}
\item[1-]\textit{pre-hereditary} if for every $S$-act $A$  and  $Y\leq X\in\Sigma_{r(A)}$, $Y\in \mathbb{R}_{r}$. \medskip
\item[2-]\textit{weakly-hereditary} if, for every $S$-act $A$ with a zero element $\theta$ and  $X\in\Sigma_{r(A)}$ with  $\theta\in X$, $X\in \mathbb{R}_{r}$. \medskip
\item[3-]\textit{zero-hereditary} if,  for every $S$-act $A$ with a zero element $\theta$ and  $Y\leq X\in\Sigma_{r(A)}$ with  $\theta\in Y$, $Y\in \mathbb{R}_{r}$. \medskip
\item[4-]\textit{pre-Kurosh} if,  for every $S$-act $A$ and  $ X\in\Sigma_{r(A)}$, $X\in \mathbb{R}_{r}$.
\end{enumerate}
\end{definition}

Figure \ref{fig} present the relation between the different kinds of radicals.

\begin{figure}[tbh]
\begin{tikzpicture}
\node [opmprocess] (7) {hereditary Kurosh-Amitsur};
\node [ below=of 7, yshift=-30pt](X)  {$\bigcirc$};
\node [opmprocess, below=of 7, xshift=-80pt] (5) {hereditary};
\node [opmprocess, below=of 5, yshift=-10pt] (4) {pre-hereditary};
\node [opmprocess, right=of 4, xshift=40pt] (6) {Kurosh-Amitsur};
\node [opmprocess, below=of 6, yshift=-10pt] (8) {pre-Kurosh};
\node [opmprocess, below=of 4, yshift=-10pt] (3) {Zero-hereditary};
\node [opmprocess, below=of 3, xshift=80pt] (2) {weakly-hereditary};
\node [opmprocess, below=of 2, yshift=-10pt] (1) {Hohenke};
\path  (5) edge[opmoutput] (4);
\path  (2) edge[opmoutput] (1);
\path  (3) edge[opmoutput] (2);
\path  (4) edge[opmoutput] (3);
\path  (7) edge[opmoutput] (5);
\path  (X) edge[opmoutput] (6);
\path  (6) edge[opmoutput] (8);
\path  (X) edge[opmoutput] (4);
\path  (X) edge[opmaffects] (7);
\path  (8) edge[opmoutput] (2);
\end{tikzpicture}
\caption{}\label{fig}
\end{figure}

\begin{theorem}\label{weak-here}
Given a  radical $r$, $c^{r}$ is weakly hereditary if and only if, $r$ is weakly hereditary.
\end{theorem}

\begin{proof}
($ \Rightarrow $) Let $A$ be an $S$-act with a zero element $\theta$ and $X$ be an $r$-class  of $A$ with $\theta\in X $. Then one can easily see that $X = c^{r}_{A}(\{\theta\})$. Now weakly heredity of $r$ implies $X=c^{r}_{A}(\{\theta\})=c^{r}_{c^{r}_{A}(\{\theta\})}(\{\theta\})=c^{r}_{X}(\{\theta\})$. Hence $X=\pi^{-1}([\{\theta\}]_{r(X/\{\theta\})})$. That is $r(X/\{\theta\})=\nabla_{X/\{\theta\}}$. Therefore $X\cong  X/\{\theta\}\in\mathbb{R}_{r}$.

 \medskip
($ \Leftarrow $)Let $r$ be a weakly hereditary radical. Then since  $c^{r}_{A}(B)/B\in \Sigma_{r(A/B)}$ and the image of $B$ under canonical epimorphism is a zero element of  $c^{r}_{A}(B)/B$,  for every subact $B$ of an $S$-act $A$. So, we have  $c^{r}_{A}(B)/B\in \mathbb{R}_{r}$. That is $r(c^{r}_{A}(B)/B)=\nabla_{c^{r}_{A}(B)/B}$.   Therefore $c^{r}_{c^{r}_{A}(B)}(B)=\pi^{-1}[B]_{r(c^{r}_{A}(B)/B)}=c^{r}_{A}(B)$ and we are done.
\end{proof}

\begin{proposition}\label{closed2act pre-Ku.}
Let $r$ be a pre-Kurosh radical and $B$ be an $r$-closed subact of an $S$-act  $A$. Then $\Sigma_{r(B)}\subseteq \Sigma_{r(A)}$. 
\end{proposition}

\begin{proof}
We know that $r(B)\leq r(A)\wedge \nabla_{B}$, for every  radical $r$ and a  subact $B$ of an $S$-act $A$. So, for every $X_{B}\in \Sigma_{r(B)}$,
there exists an $r(A)$-class $X_{A}$ such that $X_{B}\subseteq X_{A}$. Since $X_{B}$ is a subact of $A$, $X_{A}$ is a subact of $A$. So $X_{A}\in \Sigma_{r(A)}$ and, By Proposition \ref{closed2act}, we have $X_{B}\leq X_{A}\leq B$. Now since $r(X_{A})\leq r(B)\wedge \nabla_{X_{A}}$,  there exists an $r(B)$-class $C\in \Sigma_{B}$ such that $X_{A}\subseteq C$. Therefore $X_{B}=C=X_{A}$ since $\Sigma_{r(B)}$ is a set of disjoint subacts of $B$.
\end{proof}

Now since for every  Kurosh-Amitsur radical $r$ we have $r(A)=\rho_{\Sigma_{r(A)}}$, by the above proposition, one can easily see that every  Kurosh-Amitsur radical is hereditable for $r$-closed subacts. See the following corollary.

\begin{corollary}\label{k.-c.}
Let $r$ be a Kurosh-Amitsur radical and $B$ be an $r$-closed subact of an $S$-act  $A$. Then $r(B)= r(A)\wedge \nabla_{B}$.
\end{corollary}

We denote the injective hull of an $S$-act $A$ by $E(A)$ and give the following lemma.

\begin{lemma}\label{l-r-ex}
Let $r$ be a  radical and  $\chi$ be a congruence  on an $S$-act  $B$. Then  $\rho_{B} \leq \pi^{-1}(r(E(B)/\chi))$ in which $\pi_{E(B)}: E(B)\to E(B)/\chi$ is  the canonical epimorphism if and only if there exists an extension $A$ of $B$ with  $\rho_{B} \leq \pi^{-1}(r(A/\chi))$, where  $\pi_{A}: A\to A/\chi$ is the  canonical epimorphism.
\end{lemma}

\begin{proof}
$(\Rightarrow)$ It is enough to take $A=E(B)$.\\
$(\Leftarrow)$ Let $B\leq A$ and $\rho_{B} \leq \pi_{A}^{-1}(r(A/\chi))$. Then there exists a homomorphism $f: A\to E(B)$ which commutes the following diagram.
\[\xymatrix{
& B  \ar[r]^{\iota_{B}} \ar[rd]_{\iota_{E(A)}}& A \ar[d]^(.35){f} \\
&& E(B) }\]
 Now we consider the homomorphism $\overline{f}:A/\chi \longrightarrow  E(B)/\chi$, mapping each $[a]_{\chi}$ to $[f(a)]_{\chi}$.
Then we have
$\rho_{B/\chi}=\overline{f}(\rho_{B/\chi})\subseteq \overline{f}(r(A/\chi)\subseteq r(\overline{f}(A/\chi)\leq r(E(B)/\chi)$ and  so,  $\rho_{B}\leq\pi_{E(B)}^{-1}(\rho_{B/\chi})\leq\pi^{-1}( r(E(B)/\chi)).$
\end{proof}

\begin{theorem}\label{l-r-ext}
Given a subact $C$  of an $S$-act $B$, $B\leq c^{r}_{E(B)}(C)$  if and only if there exists an extension $A$ of $B$ with $B\leq c^{r}_{A}(C)$.
\end{theorem}

\begin{proof}
$(\Rightarrow)$ It is clear.\\
$(\Leftarrow)$  Let $B\leq c^{r}_{A}(C)$. Then $\rho_{B}\leq \pi_{A}^{-1}(r(A/C)$, since $c^{r}_{A}(C)=\pi_{A}^{-1}([C]_{r(A/C)})$, for the canonical epimorphism $\pi_{A}:A\to A/C$. Hence  Lemma \ref{l-r-ex} implies $\rho_{B}\leq \pi_{E(B)}^{-1}(r(E(B)/C)$ where  $\pi_{E(B)}:E(B)\to E(B)/C$ is the canonical epimorphism. But $C\leq B$. So, $\pi_{E(B)}(B)\leq [C]_{r(A/C)}$, and hence $B\leq \pi^{-1}_{E(B)}([C]_{r(A/B)})=c_{E(B)}^{r}(C)$.
\end{proof}

\begin{proposition}\label{zero}
For a  radical $r$ of  {\bf S-Act}, the following conditions are equivalent.
\begin{itemize}
\item[$(a)$] The radical $r$ is  zero-hereditary.
\item[$(b)$]  A subact $ C $ of an $S$-act $B$ is $r$-dense in $B$ if and only if there exists an extension $A$ of $B$ with $B\leq c^{r}_{A}(C)$.
\end{itemize}
\end{proposition}

\begin{proof}
$(a)\Rightarrow (b)$ To prove necessity it is enough to consider $B=A$.\\
To prove sufficiency first we note that  $B\leq c^{r}_{E(B)}(C)$, by Theorem \ref{l-r-ext}. So, $\pi_{E(B)}(B)=B/C\leq [C]_{r(E(B)/C)}$ in which $\pi_{E(B)}$ is canonical homomorphism from $E(B)$ to $E(B)/C$. Also,  $[C]_{r(E(B)/C)}\in \Sigma_{r(E(B)/C)}$ and the homomorphic image of  $C$ under $\pi_{E(B)}$ is a zero element of $B/C$. Thus, by the hypothesis, we have $r(B/C)=\nabla _{B/C}$. This means that $ C $ is $r$-dense in $B$.

 \medskip
$(b)\Rightarrow (a)$ Suppose $A\in S$-$Act$, $B\in \Sigma_{r(A)}$  and  $ C\leq B $ with $\theta\in C$. Then  $\rho_{C} \leq r(A/{\theta})$. and hence  $\{\theta\}$ is $r$-dense in $C$, by Condition $(b)$. Therefore $ r(C)=\nabla_{C} $.
\end{proof}

In the following we give a definition of intersection large subacts in a more general meaning than it is in  \cite{Weinert}.

\begin{definition}
A non-trivial subact $B$  of $A$  is called \textit{intersection large } \linebreak ($\cap$-large) in $A$ if $|B\cap X|\geqslant 2$, for all non-trivial subact $X$ of $A$.
\end{definition}

\begin{lemma}\label{cap-large}
Let  $B$ be a large subact of an $S$-act $A$. Then  $B$ is $\cap$-large in $A$.
\end{lemma}

\begin{theorem}
Let $r$ be a pre-hereditary  radical, $A\in \mathbb{S}_{r}$ and  $B$  be an $r$-dense subact of $A$. Then $B$ is $\cap$-large in $A$.
\end{theorem}

\begin{proof}
Let $X$ be a non-trivial subact of $A$. We have to show that $ |B\cap X|\geq 2$. But since $X/(B\cap X)\leq A/B\in \mathbb{R}_{r}$,  $X/(B\cap X)\in \mathbb{R}_{r}$ follows from being  pre-hereditary of $r$.  Therefore $X/(B\cap X)\ncong X$ because otherwise $X\in \mathbb{R}_{r}\cap \mathbb{S}_{r}$ which implies that $X$ is trivial $S$-act which is a contradiction. So $\rho_{B\cap X}\neq \Delta_{X}$ which means $|B\cap X|\geqslant 2$ and we are done.
\end{proof}

\begin{proposition}\label{close coproduct zero-her.}
Let $r$ be a  zero-hereditary  radical whose semisimple class is closed under coproducts  and $B$ be an $r$-dense subact of a semisimple $S$-act $A$. Then $B$ is $\cap$-large in $A$.
\end{proposition}

\begin{proof}
To prove, we show that, for every $x\in A\setminus B$, $|B\cap Sx|\geqslant 2$.  To do so, we  suppose there exists $x\in A\setminus B$ such that   $|B\cap Sx|\lneq 2$ and we get a contradiction. So let  $|B\cap Sx|\lneq 2$, then two possible cases may occur;

 {\rm (i)}  $|B\cap Sx|=0$ and  ${\rm(ii)}~|B\cap Sx|=1$.
In both cases  $(Sx\cup B)/B\in \mathbb{S}_{r}$, because, in case (i),   $(Sx\cup B)/B \cong Sx\amalg\Theta$ in which $\Theta$ is a singleton trivial $S$-act. Hence $(Sx\cup B)/B\in \mathbb{S}_{r}$, follows from the closedness of $\mathbb{S}_{r}$  under coproducts  and this fact that  $Sx, \Theta\in\mathbb{S}_{r}$. Also, in case (ii), we have $(Sx\cup B)/B \cong Sx\in \mathbb{S}_{r}$. Also since  $B\leq Sx\cup B\leq A$ and $B$ is $r$-dense in $A$. Proposition \ref{zero} implies that $(Sx\cup B)/B\in \mathbb{R}_{r}$. So $(Sx\cup B)/B\in \mathbb{S}_{r}\cap \mathbb{R}_{r}$ which means $(Sx\cup B)/B$ is a trivial $S$-act and so $(Sx\cap B)/B=B$. Therefore $Sx\leq B$ which contradicts $x\in A\setminus B$.
\end{proof}

\section{Banaschewski's condition on $r$-monomorphisms}

Because of the crucial role of Banaschewski's condition in the study of the well-behaviour of injectivy, we dedicate this short section to verify this condition concerning $r$-monomorphisms. To do so, we use the notion of essential congruence as introduced in  \cite{Wiegandt (2006)}. this notion is tightly related to the notion of essential monomorphisms which is important to study injective hull, see for example \cite{Ebrahimi (2009),Ebrahimi (2014)}. Now let us give the definition of essential congruence in \textbf{S-Act}.

\begin{definition}
A congruence $\chi$ on an $S$-act $A$ is said to be essential if
$\chi\wedge \theta\neq \Delta_{A}$, for every congruence $\theta\neq \Delta_{A}$ on $A$.
\end{definition}

\begin{definition}
A family $\{A_{i}\}_{i\in I}$ of subacts of an $S$-act $A$ is called \textit{collectively large} in $A$  if any homomorphism $g:A\rightarrow C$ whose restriction to $A_{i}$'s is a monomorphism, is itself a monomorphism.
\end{definition}

In the following we give the relation between two former defined notion.

\begin{theorem}
A family $\Sigma=\{A_{i}\}_{i\in I}$ of disjoint  subacts of an $S$-act $A$ is collectively large in $A$ if and only if the Rees congruence $\rho_{\Sigma}$ is an essential congruence  on $A$.
\end{theorem}

\begin{proof}
($ \Rightarrow $) Suppose that $\Sigma=\{A_{i}\}_{i\in I}$ is collectively large in $A$ and $\rho_{\Sigma}\wedge \chi= \Delta_A$, for some $\chi\in {\Con}(A)$. Then $\pi|_{A_{i}}:A_{i}\to A/\chi$ is a monomorphism, for every $i\in I$ where $\pi: A\to A/\chi$ is  the  canonical epimorphism. Hence  $\pi: A\to A/\chi$ is  a monomorphism and so $\chi=\Delta_{A}$.

 \medskip
($\Leftarrow$) Let $\Sigma=\{A_{i}\}_{i\in I}$ be a family  of disjoint  subacts of $S$-act $A$ such that $\rho_{\Sigma}$ is an essential congruence on $A$ and also let $g : A \to C$ be a homomorphism such that $g|_{A_{i}}$ is a monomorphism, for every $i\in I$. Then ${\ker} (g)\wedge \rho_{\Sigma}=\Delta_A$. So ${\ker} (g)=\Delta_A$ follows from essentiality of $\rho_{\Sigma}$. That is  $\{A_{i}\}_{i\in I}$ is collectively large in $A$.
\end{proof}

\begin{corollary}
A homomorphism $f : A \rightarrow B$ is an essential monomorphism if and only if the generated Rees congruence by $f(A)$, that is $\rho_{f(A)}$, is essential on $B$.
\end{corollary}

\begin{theorem}\label{essential congruence on A/kappa}
Let $A$ be an $S$-act, $\chi\in {\Con}(A)$ and $\kappa$ be a maximal congruence with $\chi\wedge \kappa = \Delta_{A}$. Then, $(\kappa \vee \chi ) / \kappa$ is an essential congruence on
$A/\kappa$.
\end{theorem}

\begin{proof}
\sloppypar First we claim that $\tau\wedge \chi\neq\Delta_A$, for every congruence $\tau/\kappa\in~ { \Con}(A/\kappa)$. This follows from the maximality of $\kappa$ with respect to $\chi\wedge\kappa=\Delta_A$, and the fact that every congruence on $A/\kappa$ is in the form of $\tau/\kappa$ in which $\tau\in {\Con}(A)$ contains $\kappa$. Therefore, for every $\tau/\kappa\in {\Con}(A/\kappa)$, there exist $x\neq y$ in $A$ such that $(x,y)\in \chi\wedge \tau$. But since $(x,y)\in \chi$ and $\chi\wedge\kappa=\Delta_A$, we have $(x,y)\notin \kappa$. So $[x]_{\kappa}\neq [y]_{\kappa}$ and $([x]_{\kappa},[y]_{\kappa})\in \tau/\kappa \wedge (\chi\vee\kappa)/\kappa$. Hence we have $(\chi\vee\kappa)/\kappa\wedge\tau/\kappa\neq\Delta_{A/\kappa}$, for every $\tau/\kappa\neq \Delta_{A/\kappa}$ in ${ \Con}(A/\kappa)$. This means that $(\chi\vee\kappa)/\kappa$ is an essential congruence on $A/\kappa$.
\end{proof}

\begin{lemma}\label{essential subact}
Let $A$ be an $S$-act,  $\chi \in {\rm \Con}(A)$  and $\kappa$ be a maximal congruence with the property  $\chi \wedge \kappa = \Delta_{A}$.  Then $ \pi|_{B}   : B \rightarrow B/\kappa|_{B}$, mapping  each $b \in B$ to $[b]_{\kappa}$,  is an isomorphism,  for every $B \in\Sigma_{\chi}$.
\end{lemma}

\begin{proof}
To prove it is enough to show that the map $\pi|_B$ is injective. Indeed, if $b\neq b'$ in $B$ then $(b,b')\in \chi$ and hence $(b,b')\notin \kappa$ follows from $\chi\wedge\kappa=\Delta_A$. That is $[b]_{\kappa}\neq [b']_{\kappa}$, for every $b\neq b'$ in $B$.
\end{proof}

\begin{lemma}\label{rho vee kappa/kappa}
 Let $A$ be an $S$-act,  $\rho_{B}$  be the Rees congruence  on the subact $B$ of $A$, and $\kappa_{B}$   be a maximal  congruence  with the property  $\rho_{B}  \wedge \kappa_{B}  = \Delta_{A}$. Then  $\pi(\rho_{B} )  = (\rho_{B}  \vee \kappa_{B} )/\kappa_{B}$,  in  which $\pi  : A  \rightarrow  A/\kappa_{B}$  is  the  canonical homomorphism.
\end{lemma}

\begin{proof}
\sloppypar To prove we show that $\pi(\rho_B)=\{([a]_{\kappa_B},[b]_{\kappa_B})\, |\, (a,b)\in \rho_B\}$ is a congruence on $A/\kappa_B$, for every subact $B$ of $A$, then Correspondence Theorem gives the result. The reflexive and symmetric properties of $\pi(\rho_B)$ easily follows from being epimorphism of $\pi$ and  the reflexive and symmetric properties of $\rho_B$. To check the transitive property suppose that $([a]_{\kappa_B},[b]_{\kappa_B}), ([c]_{\kappa_B},[d]_{\kappa_B})\in ~\pi(\rho_B)$ and $[b]_{\kappa_B}=[c]_{\kappa_B}$. Then we have the following possible cases:
\begin{enumerate}
    \item $c=d$, then $([a]_{\kappa_B},[d]_{\kappa_B})=([a]_{\kappa_B},[c]_{\kappa_B})=([a]_{\kappa_B},[b]_{\kappa_B})\in \pi(\rho_B)$;

    \item $a=b$, then $([a]_{\kappa_B},[d]_{\kappa_B}=([b]_{\kappa_B},[d]_{\kappa_B}=([c]_{\kappa_B},[b]_{\kappa_B})\in \pi(\rho_B)$;

    \item $a\neq b$ and $c\neq d$, then $a,b,c,d\in B$, since $(a,b),(c,d)\in \rho_B$. So $([a]_{\kappa_B},[d]_{\kappa_B})\in \pi(\rho_B)$.
\end{enumerate}
The compatibility of $\pi(\rho_{B})$ with the action is obvious. Hence $\pi(\rho_B)$ is a congruence on $A/\kappa$ and we are done.
\end{proof}

Now we give the Banaschewski's condition for $r$-monomorphisms, but first let us note the following definition.

\begin{definition}\label{d-r-ess}
A subact $B$ of an $S$-act $A$ is called to be \textit{$r$-large} if $B$  is both large and $r$-dense in $A$. Then we call $A$ to be \textit{$r$-essential extension} of $B$.

 Also, an $r$-monomorphism $\iota: B \rightarrow A$ is called \textit{$r$-essential monomorphism} if $\iota(B)$ in $A$  is $r$-large.
\end{definition}

\begin{theorem}[Banaschewski's $r$-condition]\label{b3}
 Given an $r$-monomorphism $f:B\to A$, there exists a homomorphism $g: A\to X$ such that  $g\circ f: B \to X$ is an $r$-essential monomorphism.
\end{theorem}

\begin{proof}
To prove, it is enough to show that  there exists a congruence $\kappa$ on the $S$-act $A$  such that  $\pi(f(B))$ is $r$-large in $A/\kappa$, for  the canonical homomorphism $\pi: A\to A/\kappa$. But, from Lemma 3.1 of \cite{Wiegandt (2006)}, we know that there exists a maximal congruence $\kappa$ on $A$  with respect to $\rho_{f(B)}\wedge \kappa = \Delta_{A}$. So $\pi(f(B))$ is large in $A/\kappa$, by Theorem \ref{essential congruence on A/kappa} and Lemma \ref{rho vee kappa/kappa}. Also $\pi(f(B))$ is an $r$-dense subact of $A/\kappa$,   by Lemma \ref{hom-dens}. Therefore $\pi(f(B))$ is $r$-large in $A/\kappa$.
\end{proof}

\section{$r$-injective $S$-acts}

In this section we discus the notion of $r$-injectivy in \textbf{S-Act}, where $r$ is a  radical, and give some properties concerning $r$-injective $S$-acts to identify this kind of injectivy. Let us begin with the following definition.
\begin{definition}\label{d-r-inj}
Let $r$ be a  radical. An $S$-act $Q$ is called \textit{$r$-injective} if, for every  $ r $-monomorphism $ \iota :A \longrightarrow B $, every homomorphism $f : A \longrightarrow Q$ can be extended  to the homomorphism $\overline{f} : B \longrightarrow Q$ thorough  $i:A \to B$, that is $f=\overline{f}
i$. Moreover  $Q$ is called \textit{orthogonal $r$-injective} if $\overline{f}$ is unique.
\end{definition}

\begin{theorem}
Let $r$ be a  radical. Then a subact $F$ of an $r$-injective $S$-act $E$ is $r$-injective if $E/F\in \mathbb{S}_{r}$.
\end{theorem}

\begin{proof}
Suppose $E/F\in \mathbb{S}_{r}$ and consider the diagram
\[
\xymatrix{
& A  \ar[r]^m \ar[d]_{f}& B \\
& F      \ar@{^{(}->}[r]      &  E \\
 }
\]
in which $m$ is an $r$-monomorphism. Then there exists a homomorphism $\overline{f}:B\to E$  which commutes the above diagram. Now consider  the homomorphism $f':B/A\to E/F$ which maps each $[b]_{A}\in B/A$ to $[\overline{f}(b)]_{E}$.  Since $B/A\in \mathbb{R}_{r}$ and $E/F\in \mathbb{S}_{r}$, and also $\mathbb{R}_{r}$ is closed under homomorphic image, $f'$ is a zero homomorphism. This implies that $\overline{f}(B)\subseteq F$. That is,  $\overline{f}:B\to F$ is a homomorphism with $\overline{f}\circ m=f$, and we are done.
\end{proof}

To give a characterization of $r$-injective $S$-acts, first we give the following lemma.
\begin{lemma}
Given  a  radical $r$,  the class $\mathbb{L}_{r}=\{A/B\ | B \text{ is } r\text{-dens in } A\} $ is the radical class of a Kurosh-Amitsur radical.
\end{lemma}
\begin{proof}
To prove, we use Lemma 2.4 of \cite{Wiegandt (2006)} and we show that $\mathbb{L}_{r}$ is  closed under homomorphic image and Rees extension, and has inductive property.

 \medskip
 The closedness of $\mathbb{L}_{r}$ under homomorphic image: since each $X \in \mathbb{L}_{r}$ is a member of $\mathbb{R}_{r}$ and has a zero element, every homomorphic image of $X$ such as $Y$ belongs to $\mathbb{R}_{r}$ and has an element such as $\theta_{0}$. So, $\{\theta_{0}\}$ is $r$-dense in $Y$. Therefore $Y/\{\theta_{0}\}\cong Y$ is in $\mathbb{L}_{r}$ and this means that  $\mathbb{L}_{r}$ is closed under homomorphic image.

 \medskip
The closedness of $\mathbb{L}_{r}$ under Rees extension:  let $A$ be an $S$-act and $\rho$ be a Rees congruence on $A$ such that $\Sigma_{\rho}\subseteq \mathbb{L}_{r}$ and $A/\rho\in \mathbb{L}_{r}$. Then  $A$ has a zero element such as $\theta_{0}$, since every $B\in\Sigma_{\rho}$  has a zero element. Also $A$ belongs to $\mathbb{R}_{r}$, since $\Sigma_{\rho}\subseteq \mathbb{L}_{r}\subseteq \mathbb{R}_{r}$ and $A/\rho\in \mathbb{L}_{r}\subseteq \mathbb{R}_{r}$. Therefore $A/\{\theta_{0}\}\cong A$ is in $\mathbb{L}_{r}$ and this means that  $\mathbb{L}_{r}$ is closed under Rees extension.

 \medskip
Inductive property: let $\{A_{i}\}_{i\in I}$ be an ascending  chain in $ \mathbb{L}_{r}$. Then $\bigcup_{i \in I}A_{i}$ has a zero element such as $\theta_{0}$ and belongs to $\mathbb{R}_{r}$. Hence $\{\theta\}$ is $r$-dense in  $\bigcup_{i \in I}A_{i}$. Therefore $\bigcup_{i \in I}A_{i}/\{\theta_{0}\}\cong \bigcup_{i \in I}A_{i}$ is in $\mathbb{L}_{r}$ and this means that  $\mathbb{L}_{r}$ has inductive property.
\end{proof}

\begin{theorem}
Given  a  radical $r$, the class of $r$-injective $S$-acts  is exactly the class of $t_{\mathbb{L}_{r}}$-injective  $S$-acts, where $t_{\mathbb{L}_{r}}$ is the induced Kurosh-Amitsur radical by $\mathbb{L}_{r}$.
\end{theorem}

\begin{proof}
One can easily see that a subact $B$ of an $S$-act $A$ is $r$-dense if and only if $A/B\in \mathbb{R}_{r}$. So  $\mathbb{L}_{r}\subseteq \mathbb{R}_{r}$ and this implies that every $t_{\mathbb{L}_{r}}$-dense subact of   $A$ is $r$-dense. Hence  every $r$-injective $S$-act is  $t_{\mathbb{L}_{r}}$-injective.
Conversely let $I$ be a  $t_{\mathbb{L}_{r}}$-injective $S$-act. Then since, for every $r$-monomorphism $m:A\rightarrow B$, $B/m(A)$ belongs to  $\mathbb{L}_{r}$, every homomorphism $f:A\rightarrow I$ can be extended to  $\overline{f}:B\rightarrow I$. Therefore $I$ is $r$-injective.

 \end{proof}

\begin{theorem}
Given a   radical $r$, every orthogonal $r$-injective $S$-act belongs to  $\mathbb{S}_{t_{\mathbb{L}_{r}}}$, where $t_{\mathbb{L}_{r}}$ is the Kurosh-Amitsur radical  given by $\mathbb{L}_{r}$.
\end{theorem}

\begin{proof}
Let $I$ be an orthogonal $r$-injective $S$-act and $I\notin \mathbb{S}_{t_{\mathbb{L}_{r}}}$. Then there exists a non-trivial homomorphism $f$ from an $S$-act $A\in\mathbb{L}_{r}$ to $I$. But since each  $A\in\mathbb{L}_{r}$ has  a zero element such as $\theta_{A}$,  $I$ has a zero element $\theta_{I}$, and hence the zero homomorphism $0_{\Theta I}: \{\theta_{A}\}\to I$, which maps $\theta_{A}$ to $\theta_{I}$ has at least two extension $f$ and the zero homomorphism $0_{\Theta I}(\theta_{A})=\theta_{I}$. This contradict orthogonally of $I$.
\end{proof}
 We end this section by expressing an interesting property of $r$-closed subacts of an $r$-injective $S$-act wherewith we shall give  a characterization of $r$-injective $S$-acts in Section $7$.
\begin{theorem}\label{inj-den}
Let $I$ be an $r$-injective $S$-act and $A$ be an $r$-dense subact of an $S$-act $B$. Then  the image of  $B$ under every extension $\overline{f}:B\to I$ of a homomorphism $f: A\to I$ is a subact of $c_{I}^{r}(f(A))$.
\end{theorem}

\begin{proof}
Let $A$ be  $r$-dense subact  of an $S$-act  $B$ and $f:A\rightarrow I$ be a homomorphism. Then there exists $\overline{f}:B\to I$ which commutes the following rectangle.

\[\xymatrix{
 A  \ar@{^{(}->}[rr]  \ar[d]_{f}  && B \ar[d]^(.35){\overline{f}} \ar@{^{}-->}[dl]_(.55){\overline{f}}  \\
 f(A) \ar@{^{(}->}[r]
 &c_{I}^{r}(f(A)) \ar@{^{(}->}[r]^{} & I }\]
  But we have by the third property of a closure operator $\overline{f}(c_{B}^{r}(A))\leq c_{I}^{r}(\overline{f}(A))$. Now since $c_{B}^{r}(A)=B$ and $\overline{f}(A)= f(A)$, so $\overline{f}(B)\leq c_{I}^{r}(f(A))$ and we are done.
\end{proof}

One can easily get the following corollary from the above theorem.

\begin{corollary}\label{p-inj}
 {\rm (i)} Let  $A$ be an $S$-act and  $I$ be an $r$-injective extension of $A$. Then $c^{r}_{I}(A)$ is $r$-injective.

{\rm (ii)} Let $E(A)$ be an injective hull of an $S$-act $A$. Then $c^{r}_{E(A)}(A) $ is $r$-injective.

{\rm (iii)}An $S$-act $A$ is $r$-injective if and only if  every $r$-closed subact of $A$  is $r$-injective.
\end{corollary}

\section{The well-Behaviour of $r$-injectivy}

Different sets of conditions are sufficient, although not always necessary, for the well-Behaviour of injectivy. The crucial conditions to verify whether injectivy is well-Behavior are so-called Banaschewski's condition, which is given in the previous section, $r$-transferability condition, and Direct limit condition, see \ref{b5}.
 In this section to verify the well-behaviour of $r$-injectivy, for a given   radical $r$, we first  check these conditions.

\begin{lemma}[$r$-transferability condition]\label{b4}
The category {\bf S-Act} satisfies the  \linebreak $r$-transferability property.  That is, every diagram

\[\xymatrix{
& A  \ar[r]^m \ar[d]_{f}& B  \\
& C  \\
 }\]
with the $r$-monomorphism $m$ can be completed to a commutative square as follows in which $u$ is an $r$-monomorphism.
\[\xymatrix{
& A  \ar[r]^m \ar[d]_{f}& B\ar[d]_{v}  \\
& C  \ar[r]^u &  D\\
 }\]
\end{lemma}

\begin{proof}
Consider $D=(B\setminus m(A))\dot{\cup}C$ together with the action
\[
s.x=\begin{dcases}
s*'x & x\in C \\
s*x & x\in B\setminus m(A) \text{ and } s*x\in B\setminus m(A)\\
f(m^{-1}(s*x)) & x\in B\setminus m(A) \text{ and } s*x\in m(A)
\end{dcases}
\]
in which $*$ is the action of $B$ and $*'$ is the action of $C$. Clearly,
\[\begin{matrix}
 v:&B& \longrightarrow & D \qquad \qquad \qquad\qquad\ \ \\
& b & \mapsto  & v(b)=\begin{dcases}
b & b\in B\setminus A\\
f(b) & b\in A
\end{dcases} \\
 \end{matrix}\]
and the inclusions map $u:C\rightarrow D$ makes  the following diagram commutative.
\[\xymatrix{
& A  \ar[r]^m \ar[d]_{f}& B\ar[d]_{v}  \\
& C  \ar[r]^u &  D\\
 }\]
Also, since $D/u(C) \cong B/m(A)$ and $m(A)$ is $r$-dense in $B$, $u(C)$ is $r$-dense in $D$.
\end{proof}

  We recall that a directed family of $S$-acts  is a family $(A_{i})_{i\in I}$ of $S$-acts indexed by an up-directed set $(I, \leq )$ endowed by a family $(f_{ij} : A_{i} \to A_{j} )_{i\leq j\in I}$ of monomorphisms such that given $i\leq j \leq k\in I$ we have $f_{jk}\circ f_{ij}=f_{ik}$, also $f_{ii}=id_{A_{i}}$, for every $i\in I$. Note that the direct limit of a directed family $((A_{i})_{i\in I} , (f_{ij} )_{i\leq j\in I} )$ in {\bf S-Act} is given as $\underrightarrow{lim} (A_{i})_{i\in I} = \coprod_{i\in I} A_{i} /\chi$,
 where the congruence $ \chi $ is given by $a_{i} \chi a_{j}$ if and only if there exists $k\geq i,j$ such that $u_{k}f_{ik}(a_{i}) = u_{k} f_{jk}(a_{j})$ in which each $u_{i} : A_{i}\to \coprod_{i\in I} A_{i}$ is an injection map of the coproduct.

To establish the direct limit condition for $r$-injectivy , or for short $r$-direct limit, we need the  following lemma.
\begin{lemma}\label{...}
Let $\mathbb{R}$ be a subclass of {\bf S-Act} which  is closed under homomorphic image and Rees congruence extension. Then $\mathbb{R}$ is a radical class of a  radical if and only if $\underrightarrow{lim} (A_{i})_{i\in I}\in \mathbb{R}$, for every  directed family  $((A_{i})_{i\in I} , (f_{ij} )_{i\leq j\in I} )$ of  $\mathbb{R}$.
\end{lemma}

\begin{proof}
$(\Rightarrow)$
Let $r$ be a  radical and $((A_{i})_{i\in I} , (f_{ij} )_{i\leq j\in I} )$ be a directed family in $\mathbb{R}_{r}$. Then since $\mathbb{R}_{r}$ is  closed under homomorphic image, $\pi\circ u_{i}(A_{i})$ is a radical subact of $\underrightarrow{lim} (A_{i})_{i\in I}$, for the epimorphism $\pi:\coprod_{i\in I} A_{i} \to \underrightarrow{lim} (A_{i})_{i\in I}$ and  every $i\in I$. So, by Remark \ref{•}, there exists $X_{i}\in \Sigma_{r(\underrightarrow{lim} (A_{i})_{i\in I})}$ such that $\pi\circ u_{i}(A_{i})\leq X_{i}$, for every $i\in I$.

Now we show  that for a fixed $j_{0}\in I$, $X_{i}=X_{j_{0}}$, for every $i\in I$. Because, for every $i\in I$, there exist $k\in I$ with $f_{ik}(A_{i}), f_{j_{0}}(A_{j_{0}})\leq A_{k}$. So we have $\pi\circ u_{i}(A_{i}),\pi\circ u_{j_{0}}(A_{j_{0}}), \pi\circ u_{k}(A_{k})\leq X_{k}$. Therefore $X_{i}=X_{j_{0}}=X_{k}$ follows from this fact that $\Sigma_{r(\underrightarrow{lim} (A_{i})_{i\in I})}$ consist of some disjoint subacts of $\underrightarrow{lim} (A_{i})_{i\in I}$ and hence $\underrightarrow{lim} (A_{i})_{i\in I}=\bigcup_{i\in I} \pi\circ u_{i}(A_{i})=X_{j_{0}}\in \mathbb{R}_{r}$.

$ (\Leftarrow) $ Conversely, let $R$ be a subclass of $S$-acts which is closed under homomorphic image and Rees congruence extension. Then since every chain in $\mathbb{R}$ is  a directed family, $\mathbb{R}$ has the inductive property. So, by Theorem 2.4 of \cite{Wiegandt (2006)}, $\mathbb{R}$ is a radical class of a radical.
\end{proof}

\begin{definition}
 A  directed family  $\mathfrak{D}=((A_{i})_{i\in I} , (f_{ij} )_{i\leq j\in I} )$ of $S$-acts  is called \textit{$r$-directed}  if  each $f_{ij}$ is  an $r$-monomorphisms, for every ${i\leq j\in I}$.
\end{definition}

 \begin{theorem}[ $r$-direct limit condition]\label{b5}
Let $I$ be an up-directed set with the first element $0$ and $((A_{i})_{i\in I} , (f_{ij} )_{i\leq j\in I} )$ be an  $r$-directed family of $S$-acts indexed by $I$. Then  $\pi\circ u_{i}$ is an $r$-monomorphism, where $u_{i}:A_{i}\to \coprod_{i\in I}A_{i}$ is the injection map, for every $i\in I$, and  $\pi: \coprod_{i\in I}A_{i}\to \underrightarrow{lim} (A_{i})_{i\in I} $ is the canonical epimorphism.
\end{theorem}
  \begin{proof}
  Given an $r$-directed family $\mathfrak{D}=((A_{i})_{i\in I} , (f_{ij} )_{i\leq j\in I} )$ of $S$-acts, we get the $r$-directed family $((A_{i}/f_{0i}(A_{0}))_{i\in I} , (\overline{f}_{ij} )_{i\leq j\in I} )$ in $\mathbb{R}_{r}$.  So, Lemma \ref{...} implies that $ \underrightarrow{lim} (A_{i}/f_{j_{0}i}(A_{j_{0}}))_{i\in I}\in \mathbb{R}_{r}$.

  But since
  \begin{align*}
  \underrightarrow{lim} (A_{i}/f_{j_{0}i}(A_{0}))_{i\in I}& \cong   \bigcup_{i\in I} \pi'\circ u_{i}(A_{i}/f_{0i}(A_{0}))\\
  &  \cong ( \bigcup_{i\in I} \pi\circ u_{i}(A_{i}) )/( \pi\circ u_{0}(A_{0}))
  \end{align*}
  where $\pi:\coprod_{i\in I}A_{i}\to \underrightarrow{lim} (A_{i})_{i\in I} $ and $\pi':\coprod_{i\in I}(A_{i}/f_{0i})\to \underrightarrow{lim} (A_{i}/f_{0i}(A_{0}))_{i\in I}    $  are the canonical epimorphisms,  $( \bigcup_{i\in I} \pi\circ u_{i}(A_{i}) )/ \pi\circ u_{j}(A_{j})\in \mathbb{R}_{r}$ follows from the closedness of $\mathbb{R}_{r}$ under homomorphic image, for every $j\in I$. That is  $\pi\circ u_{j}(A_{j})$ is $r$-dense in $ \underrightarrow{lim} (A_{i})_{i\in I}$,  for every $j\in I$ .
  \end{proof}

\begin{remark}
Now, as it is mentioned in \cite{ban}, in the present of conditions $B_{1}$-$B_{6}$, which are stated as follows, we have the well-Behaviour of $r$-injectivy.

$B_{1}$ - The class of $r$-monomorphisms is  composition closed. Because $c^{r}$ is an idempotent closure operator, see Section 2.4 of \cite{Tholen (1995)}.

$B_{2}$ -  The class of $r$-monomorphisms  is trivially isomorphism closed and left regular; that is, for $f \in\mathcal{M}$ with $fg = f$ we have $g$ is an isomorphism.

$B_{3}$ - Banaschewski's $r$-condition, see Theorem \ref{b3}.

$B_{4}$ - \textbf{S-Act} satisfies $r$-transferability conditions, see Lemma \ref{b4}.

$B_{5}$ - \textbf{S-Act} has $r$-direct limit of well ordered direct systems, See Theorem \ref{b5}.

$B_{6}$ - \textbf{S-Act} is $r^{*}$-cowell powered; that is for every $S$-act $A$, the  class
$$\{ m:A\rightarrow B \ | \ B\in \textbf{S-Act},\ m \text{  is an }r\text{-essential monomorphism.}\},$$
 up to isomorphism, is a set. It is trivial.
\end{remark}

\section{Bear criterion for $r$-injectivy}\label{6}

An important point of study in injectivy is to investigate where there is any relation between the desired injectivy and injectivy with respect to another subclass of monomorphisms, the result of which may be called the Bear type criterion. In this section we give the counterpart of Bear-Skornjakov criterion for $r$-injectivy. We also give another criterion to characterize the weakly injective $S$-acts.
We also give a Bear criterion for injective $S$-act in corollary \ref{cor. inj.}.
\begin{theorem}
Let $r$ be a  radical whose  radical class $\mathbb{R}_{r}$ is closed under coproduct. Then

{\rm (i)} every $r$-injective $S$-act contains a zero.

{\rm (ii)} The product $\prod_{i\in I} Q_{i}$ is $r$-injective if and only if $Q_{i}$ is an $r$-injective $S$-act, for all $i\in I$.
\end{theorem}

\begin{proof}
One can easily  prove the part (ii). To prove part (i), first  we note that $A$ is $r$-dense in $A\amalg \Theta$. Now the result is immediately follows from the following completed commutative diagram, by  $g:A\amalg\Theta\to A$.

\[\xymatrix{
& A  \ar[r]^(.35)\subseteq \ar[d]_{id_{A}}& A\coprod \Theta\ar[ld]^{g}  \\
& A.  \\
 }\]

\end{proof}


\begin{theorem}[Bear-Skornjakov]\label{Bear}
Let $  r$  be a zero-hereditary  radical of {\bf S-Act} and $Q$ be an $S$-act with a zero element $\theta$. Then  $ Q $ is $ r $-injective if and only if each homomorphism $ f:A_{0}\longrightarrow Q $ in which $A_{0}$ is an $ r $-dense subact of a cyclic $S$-act  $A$  can be extended to $A$.
\end{theorem}

\begin{proof}

To prove  it is enough to show that injectivy with respect to $r$-dense subacts of cyclic $S$-acts implies $r$-injectivy, to do so,  we follow the standard prove  of Skornjakov. So assume   $Q$ is  an $S$-act with a zero which satisfies the hypothesis and consider the following diagram

  \begin{align} \tag{$*$}\label{diag}
\xymatrix{
& B  \ar@{^{(}->}[r]^(.35){} \ar[d]_{f}& A  \\
& Q  \\
 }
\end{align}
 in which $B$ is  $r$-dense in $A$. Then  we take the poset
 \[T=\{h:C\to Q\ | \  B\leq C\leq A, \text{ and } h|_{B}=f\}\]
  together the partial order
 \[h_{1}\leq h_{2} \Leftrightarrow \Dom({h_{1}}) \leq \Dom({h_{2}})\text{ and } h_{2}|_{\Dom(h_{1})}=h_{1}\]
 But $\Dom(h)$ is $r$-dense subact of $A$, for every $h\in T$, because $A/\Dom(h)$ is homomorphic image of $A/A_{0}$ and $\mathbb{R}_{r}$ homomorphically closed.  Also one can easily see that every ascending chain $\{h_{i}: C_{i}\to Q\}_{i\in I}$ of $(T,\leq)$ has  the upper bond $h:\bigcup_{i\in I}C_{i}\to Q$ with $h(x)=h_{i}(x)$; where $x\in \Dom(h)$. Hence $T$ has a maximal element such as $h:A_{1}\to Q$, by the Zorn's lemma. Now we show that $A=A_{1}$. To do so, suppose on the  contrary that $A_{1}\lneq A$. Then there exists $a\in A\setminus A_{1}$ for which  we define $D_{a}=A_{1}\cap Sa$. If $D=\emptyset$, then
 \[ \begin{matrix}
 &\overline{f}:& A &\longrightarrow & Q \hspace{25mm} &\\
 & & a & \mapsto & \begin{dcases}
h(a)& a\in  A_{1}\\
\theta & a\in A\setminus A_{1}
\end{dcases} &\\
\end{matrix} \]
is an extension of  $f$ which commutes the diagram \eqref{diag} and we get the result.

If $D\neq \emptyset$ then, $D$ is an $r$-dense subact of $Sa$. Because kernel of the homomorphism  $k:Sa\to A/A_{1}$ defined by $k(sa)=sa/A_{1}$  is $\rho_{D}$. So  Homomorphism Theorem for $S$-acts implies that $Sa/D$ is isomorphic to a subact $H$ of $A/A_{1}$. Now since $r$ is a zero hereditary  radical and $H$ is a subact with a zero element of the radical $S$-act $A/A_{1}$, we have $r(Sa/D)\cong r(H)=\nabla_{H}\cong \nabla_{Sa/D}$.

Therefore there exists an extension $\overline{g}:Sa\to Q$ of the homomorphism $g: D \to Q$ defined by $g(sa)=h(sa)$, for every $sa\in D$. Thus this means that
  \[ \begin{matrix}
 &\overline{h}:& A_{1}\cup Sa &\longrightarrow & Q \hspace{25mm} &\\
 & & x & \mapsto & \begin{dcases}
h(x)& x\in  A_{1}\\
s\overline{g}(a) & x=sa\in Sa
\end{dcases} &\\
\end{matrix} \]
is an extension of $h$ and it contradicts  the maximality of $h$. So $A_{1}=A$ and we are done.
\end{proof}

\begin{corollary}[]\label{C-Bear}
Let $  r$  be a zero-hereditary  radical of {\bf S-Act} and $E$ be an $S$-act with a zero element $\theta$. Then  $ E $ is $ r $-injective if and only if it is injective with respect to the $r$-large monomorphisms into cyclic $S$-acts.
\end{corollary}

\begin{proof}
One way is clear. To prove converse, using  Theorem \ref{Bear}, we show that every $S$-act with a zero which satisfies the hypothesis is an injective $S$-act with respect to $r$-monomorphisms  into the cyclic $S$-acts. To do so, consider the following  diagram
\[\xymatrix{
& B  \ar[r]^(.45){m} \ar[d]_{f}& C  \\
& E  \\
 }\]
 in which $m$ is an $r$-monomorphism  and $C$ is a cyclic $S$-act. Then, by Theorem \ref{b3}, $m:B\to C$ can be extend to an $r$-large monomorphism $g\circ m :B\to C\to A$. Now existence of a homomorphism $\overline{f}: A\to E$ with $\overline{f}\circ m=f$ follows from hypothesis. Hence we get $\overline{f}|_{C}:C\to E$ which completes the designed diagram.
\end{proof}

\begin{theorem}
Given a hereditary  radical $r$, a semisimple $S$-act $I$ is weakly injective if and only if it is injective relative to all inclusions  into $S/r(S)$.
\end{theorem}

\begin{proof}
$(\Rightarrow)$ For an arbitrary  weakly injective $S$-act  $I$, consider $f'$ to be a homomorphism from a subact $K/r(K)$ of $S/r(S) $ to $I$. Then , by the hypothesis, there exists an extension homomorphism  $\overline{f}:S\to I$ for $f'\circ\pi_{K}=f$ which commutes the left triangle of the following  diagram,
  \[\xymatrix@!0{
  &&K\ar@{^{(}->}[rrr]^(.35){} \ar[lld]_(.55){\pi_{K}} \ar[ddd]_(.55){f} |!{[rd];[d]}\hole  &&& S \ar[lld]_(.55){\pi_{S}} \ar[lllddd]^(.55){\overline{f}}
  \\
   ~~ K/r(K) \ar@{^{(}->}[rrr]   \ar[ddrr]_(.5){f'} &&&  S/r(S) \ar@{.>}[ddl]_(.5){\widehat{f}}
    \\
 \\
               && I &
 }\]
in which $\pi_{S}$ and $\pi_{K}$ are the canonical epimorphisms.  But the property of the  radical, implies $\overline{f}(r(S))\leq r(I)=\Delta_{I}$ and hence $r(S)\leq {\ker}(\overline{f})$. So,  Homomorphism Theorem  for $S$-acts, implies that there exists a homomorphism $\widehat{f}$ from $S/r(S)$ to $I$ which  completes the above diagram and we are done.

  $(\Leftarrow)$ Let $I\in \mathbb{S}_{r}$ be injective relative to all inclusions  into $S/r(S)$, and $f$ be a homomorphism from a left ideal $K$ to $I$. Then by $f(r(K))\subseteq r(I)=\Delta_{I}$ we have $r(K)\leq \ker(f)$. Hence Homomorphism Theorem for $S$-acts implies the  existence of a homomorphism $f'$ from $K/r(K)$ to $I$ such that $f=f'\pi_{K}$ where $\pi_{K}: K\to K/r(K)$ is the canonical epimorphism. Now, by the hypothesis, there exists  a homomorphism $\widehat{f}$ from $S/r(S)$ to $I$ which   commutes the bottom triangle of the following diagram,
 \[\xymatrix@!0{
  &&K\ar@{^{(}->}[rrr]^(.35){} \ar[lld]_(.55){\pi_{K}}  \ar[ddd]_(.55){f} |!{[rd];[d]}\hole &&& S \ar[lld]_(.55){\pi_{S}} \ar[lllddd]^(.55){\overline{f}}
  \\
   ~~ K/r(K) \ar@{^{(}->}[rrr]   \ar[ddrr]_(.5){f'} &&&  S/r(S) \ar[ddl]_(.5){\widehat{f}}
   \\
 \\
               && I &
 }\]

  where $\pi_{S}$ and $\pi_{K}$ are the  canonical epimorphisms. Now    $\overline{f}=\widehat{f}\circ \pi_{S}$  is an extension of $f$ and commutes the desired diagram, meaning that $I$ is weakly injective.
\end{proof}


\section{$r$-injectivy for a Kurosh-Amitsur radical}\label{7}

In this section we discuss $r$-injectivy when $r$ is a Kurosh-Amitsur radical rather than a  radical to improve the results hereof. We then construct an Kurosh-Amitsur radical $r_{G}$ whose associated $r_{G}$-injective $S$-acts are exactly injective $S$-act.
Throughout this section, we assume that  $E(A)$ and $E_{r}(A)$ are   respectively,  the usual injective hull and  $r$-injective hull of the $S$-act $A$.

 \begin{proposition}\label{k. E_r}
 Let $r$ be a Kurosh-Amitsur radical. Then $E_{r}(A)=c^{r}_{E(A)}(A)$, for every $A\in { \bf S\textbf{-}Act}$.
\end{proposition}

\begin{proof}
Given an $S$-act $A$, then considering $A$ as a subact of $E(A)$, three possible cases may occur:

Case (i): $A$ is $r$-dense in $E(A)$, that is $E(A)=c_{E(A)}^{r}(A)$. Then $E(A)$ is a maximal $r$-essential extension of $A$. Therefore $E(A)=c_{E(A)}^{r}(A)$ is $r$-injective hull of $A$.

Case (ii): $A$ is $r$-closed in $E(A)$, that is $A=c_{E(A)}^{r}(A)$. Then $A$ is not $r$-dense in any extension of itself, by lemma \ref{l-r-ext}. and hence $E_{r}(A)=A=c_{E(A)}^{r}(A)$.

Case (iii): $A<c_{E(A)}^{r}(A)<E(A)$. Then, since  $c^{r}$ is weakly hereditary, by Lemma \ref{weak-here},  $A$ is  $r$-dense in $c^{r}_{E(A)}(A)$.  Also $c^{r}_{E(A)}(A)$ is an $r$-essential extension of $ A$ since $A\leq c^{r}_{E(A)}(A)\leq E(A)$. So, to prove, it is enough we verify the maximality of $c^{r}_{E(A)}(A)$ among all $r$-essential extensions of $ A$. To do so, let $B$ be an $r$-essential extension of $ A$ with $c^{r}_{E(A)}(A)\leq B$. Then  $B/A\in \mathbb{R}_{r}$ and $c^{r}_{E(A)}(A)/A\leq B/A \leq E(A)/A$. Thus $\nabla_{B/A}=r(B/A)\leq r(E(A)/A)$. Hence, by Remark \ref{•},   $C/A\in \Sigma_{r(E(A)/A)}$  exists such that $c^{r}_{E(A)}(A)/A\leq B/A\leq C/A$.
But since the subacts in  $\Sigma_{r(E(A)/A)}$ are disjoint and $c^{r}_{E(A)}(A)/A\in \Sigma_{r(E(A)/A)}$  ($c^{r}_{E(A)}/A=[A]_{r(E(A)/A)}$), we have  $c^{r}_{E(A)}(A)/A= B/A= C/A$. Thus $c^{r}_{E(A)}(A)= B$. That is $c^{r}_{E(A)}(A)$ is the maximal $r$-essential extension of $ A$.
\end{proof}
We use the above  proposition to give a characterization of the $r$-injective $S$-acts, see the following corollary.
\begin{corollary}  \label{r-in. clo. inj}
Given  a Kurosh-Amitsur radical $r$, an $S$-act $A$ is  $r$-injective if and only if $A$ is an $r$-closed subact of $E(A)$.
\end{corollary}

\begin{proof}
($ \Rightarrow $) If $A$ be an $r$-injective $S$-act, then  $E_{r}(A)=A$. But since $E_{r}(A)=c^{r}_{E(A)}(A)$, by Proposition \ref{k. E_r}, $A=c^{r}_{E(A)}(A)$. That is,  $A$ is  $r$-closed in $E(A)$.

 \medskip
($ \Leftarrow $)  Immediately follows from Lemma \ref{inj-den}.
\end{proof}
In the following we give a characterization of the hereditary Kurosh-Amitsur radicals by injective hull and $r$-injective hull. But first we recall the lemma bellow  from  \cite{Haddadi} which is used in the sequel.
\begin{lemma}\label{pair kurosh-amitsur}
A pair $(\mathbb{R}, \mathbb{S})$ of subclasses of $S$-acts is the radical class and the semisimple class of a Kurosh-Amitsur radical $r$ if and only if
\begin{enumerate}
\item  $\mathbb{R}\cap \mathbb{S}$ consists of trivial $S$-acts,
\item  $\mathbb{R}$ is homomorphically closed,
\item  $\mathbb{S}$ is closed under taking subacts,
\item  every $S$-act $A$ has  an $\mathbb{R}$-$system$ such as $\Sigma$ whose Rees factor, $A/\rho_{_{\Sigma}}$, belong to $\mathbb{S}$.
\end{enumerate}
\end{lemma}

\begin{theorem}\label{6k.h.}
For a Kurosh-Amitsur radical $r$ of  {\bf S-Act}, the following conditions are equivalent.
\item[\rm (1)] The radical $r$ is hereditary.
\item[\rm (2)]  Given an $S$-act $B$, the homomorphic image $B$ under a homomorphism $f$ is a radical $S$-act  if and only if there exists an extension  $A$ of $B$ such that  $\nabla_{B}\subseteq \pi^{-1}(r(A/\ker(f)\vee\Delta_{A}))$ where $\pi: A\to A/(\ker(f)\vee\Delta_{A})$ is the canonical epimorphism.
\item[\rm (3)] The radical class $\mathbb{R}_{r}$ is closed under taking subacts.
\item[\rm (4)] The semisimple class $\mathbb{S}_{r}$ is closed under $r$-injective hulls.
\item[\rm (5)] The semisimple class $\mathbb{S}_{r}$ is closed under injective hulls.
\item[\rm (6)] The semisimple class $\mathbb{S}_{r}$ is closed under essential Rees extensions.
\end{theorem}

\begin{proof}
$(1) \Rightarrow (2)$ Necessity: Follows from  Homomorphism Theorem, for $S$-acts, when we take $B=A$.

Sufficiency: From Lemma \ref{l-r-ex}, we know that the hypothesis implies\linebreak $\nabla_{B} \subseteq \pi^{-1}(r(E(B)/(\ker(f)\vee\Delta_{E(B)})))$, where $E(B)$ is the injective hull of $B$ and $\pi:E(B)\to E(B)/(\ker(f)\vee \Delta_{E(B)})$ is the canonical epimorphism. So, we have $\nabla_{B/\ker(f)} \subseteq r(E(B)/(\ker(f)\vee\Delta_{E(B)})$. Hence
\[r(f(B))\cong r(\frac{B}{\ker(f)})=r(\frac{E(B)}{\ker(f)\vee\Delta_{E(B)}})\wedge \nabla _{\frac{B}{\ker(f)}}=\nabla_{\frac{B}{\ker(f)}}\cong\nabla _{f(B)}\]
 since $r$ is hereditary.

 \medskip
$(2)\Rightarrow (3)$ Suppose $A\in \mathbb{R}_{r}$ and $i: B\to A $ is the inclusion map. Then \linebreak $\nabla_{B}= \nabla_{i(B)}\subseteq \nabla_{A}=r(A)= r(A/\Delta_{A})$. Hence hypothesis implies that  $B\in \mathbb{R}_{r} $.

 \medskip
$(3)\Rightarrow (4)$  To prove, we show that  the $r$-injective hull $E_{r}(A)$ of each semisimple $S$-act $A$ is a semisimple $S$-act. Indeed,  The  largeness of $A$ in $E_{r}(A) $ implies that \linebreak$A \cap X\neq \emptyset$, for every non-trivial subact  $X\in \Sigma_{r(E_{r}(A))}$. But we know that $A\in \mathbb{S}_{r}$, $X\in \mathbb{R}_{r}$ and both $\mathbb{S}_{r}$ and $\mathbb{R}_{r}$ are closed under taking subacts. So, we have $A \cap X\in \mathbb{R}_{r}\cap \mathbb{S}_{r}$. Hence $A\cap X$ is a trivial $S$-act since $\mathbb{R}_{r}\cap \mathbb{S}_{r}$  consists of the trivial $S$-acts. Thus $X$ is a trivial $S$-act which means $\Sigma_{r(E_{r})}=\emptyset$. Therefore $E_{r}(A)\in \mathbb{S}_{r}$ since $r$ is a Kurosh-Amitsur radical.

\medskip
$(4) \Rightarrow (5)$  To prove, we show that  the injective hull $E(A)$ of each semisimple $S$-act $A$ is a semisimple $S$-act. Indeed,  The  largeness of $E_{r}(A)$ in $E(A) $ implies that $E_{r}(A) \cap \pi^{-1}(X)\neq \emptyset$, for every non-trivial subact  $X\in \Sigma_{r(E(A)/E_{r}(A))}$  and the canonical epimorphism $\pi: E(A)\to E(A)/E_{r}(A)$. Thus $X= [E_{r}(A)]_{r(E(A)/E_{r}(A)}$, for every non-trivial subact  $X\in \Sigma_{r(E(A)/E_{r}(A))}$,  since $X$ and $[E_{r}(A)]_{r(E(A)/E_{r}(A)}$ are $r(E(A)/E_{r}(A))$-classes. But $[E_{r}(A)]_{r(E(A)/E_{r}(A))}$ is  singleton since $E_{r}(A)$ is $r$-closed in $E(A)$. So,  $X$ is a trivial $S$-act. Thus $\Sigma_{r(E(A)/E_{r}(A))}$ is empty, and hence $E(A)/E_{r}(A)$ belongs to $\mathbb{S}_{r}$ since $r$ is a Kurosh-Amitsur radical. Also $E_{r}(A)\in \mathbb{S}_{r}$, by hypothesis. Therefore $E(A)\in \mathbb{S}_{r}$ follows form  the closedness of $\mathbb{S}_{r}$  under Rees congruence extension. This means that $\mathbb{S}_{r}$ is closed under injective hulls.

 \medskip
$(5)\Rightarrow (6)$ Suppose $\rho$ is an essential  Rees congruence on an $S$-act $A$ with $\Sigma_{\rho}\in \mathbb{S}_{r}$. We Show that $A\in \mathbb{S}_{r}$. To do so, we contrary assume on the $A\notin \mathbb{S}_{r}$. Then $r(A)\neq \Delta_{A}$ and $\rho\cap r(A)\neq \Delta_{A}$ follows from   essentiality of $\rho$. Thus there exists a non-trivial subact $B\leq C\in \Sigma_{r(A)}$ such that $\rho_{B}\leq  \rho\cap r(A)$ since $r(A)$ and $\rho$ are Rees congruence. The closedness of $\mathbb{S}_{r}$ under taking subacts implies that $B\in \mathbb{S}_{r}$, and we have $E(B)\in \mathbb{S}_{r}$, by hypothesis. Now consider the following commutative diagram.
\[\xymatrix{
& B  \ar[r]^(.35)\subseteq \ar[d]_{\subseteq}& C \ar[ld]^f  \\
& E(B)  \\
 }\]
 We note that $f(C)\in \mathbb{S}_{r}\cap \mathbb{R}_{r}$, since $C\in \mathbb{R}_{r}$ and $E(B)\in \mathbb{S}_{r}$.
  Thus, by Lemma \ref{pair kurosh-amitsur},  $f(C)$ is a trivial $S$-act. The commutativity of the above diagram implies that $B$ is trivial and this is a contradiction. Therefore $\mathbb{S}_{r}$ is closed under essential Rees extension.

 \medskip
$(6)\Rightarrow (1)$ Proposition 3.3 of \cite{Wiegandt (2006)} implies that $\mathbb{R}_{r}$ is closed under taking subacts and this implies (1) by Proposition 4.1 of \cite{Wiegandt (2006)}.

\end{proof}

\begin{theorem}
Given  a Kurosh-Amitsur radical $r$, an $r$-injective $S$-act $A$ is a semisimple $S$-act if and only if $E(A)$ is a semisimple $S$-act.
\end{theorem}

\begin{proof}
$(\Rightarrow)$ Let $A$ be a semisimple $r$-injective $S$-act. Then $A$ is $r$-closed in $E(A)$, by Corollary \ref{r-in. clo. inj}. Also $\Delta_{A}=r(A)=r(E(A))\wedge \nabla_{A}$, by Theorem \ref{k.-c.}.  Hence $\rho_{A\cap B}\leq r(E(A))\wedge \nabla_{A}=\Delta_{A}$, for all non-trivial subact $B\in \Sigma_{r(E(A))}$. Thus $|A\cap B|\leq 2$. But, for all non-trivial subact $B\in \Sigma_{r(E(A))}$, we have  $|A\cap B|\geq 2$  since  $A$ is large in $E(A)$. So every $B\in \Sigma_{r(E(A))}$ is a trivial $S$-act. Hence $ \Sigma_{r(E(A))}=\emptyset$. Therefore $r(E(A))=\Delta_{E(A)}$ since $r$ is a Kurosh-Amitsur radical. That is $E(A)\in \mathbb{S}_{r}$.

 \medskip
$(\Leftarrow)$ It follows from the closedness of $\mathbb{S}_{r}$ under taking subact.
\end{proof}

\begin{theorem}
Let $r$ be a  Kurosh-Amitsur radical. Then
\item[{\rm(1)}] the radical class $\mathbb{R}_{r}$ is closed under $r$-injective hulls.
\item[{\rm(2)}]  Every $B\in\Sigma_{r(A)}$ is  $r$-injective  if $A$ is an $r$-injective $S$-act.
\end{theorem}

\begin{proof}
\textbf{(1)}  By Proposition   \ref{k. E_r}, we have $E_{r}(A)=c^{r}_{E(A)}(A)$, for $A\in$ \textbf{S-Act}. So, to prove, it is enough to show that  $c^{r}_{E(A)}(A)\in \mathbb{R}_{r}$, for every $A\in \mathbb{R}_{r}$. But,  since $c^{r}$ is weakly hereditary, see Lemma \ref{weak-here}, we have $c^{r}_{E(A)}(A)=c^{r}_{c^{r}_{E(A)}(A)}(A)=\pi^{-1}([A]_{r(c^{r}_{E(A)}(A)/A)})$, for the canonical epimorphism $\pi:c^{r}_{E(A)}(A)\to c^{r}_{E(A)}(A)/A)$. Hence $c^{r}_{E(A)}(A)/A=[A]_{r(c^{r}_{E(A)}(A)/A)}$. Also since $\pi(A)$ is a zero element of $c^{r}_{E(A)}(A)/A$,  $[A]_{r(c^{r}_{E(A)}(A)/A)}\in \Sigma_{r(c_{E(A)}^{r}(A)/A)}$. Therefore  $r(c^{r}_{E(A)}(A)/A)=\nabla_{c^{r}_{E(A)}(A)/A}$. That is  $c^{r}_{E(A)}(A)\in \mathbb{R}_{r}$.

 \medskip
\textbf{(2)}  Let $A$ be an $r$-injective $S$-act and  $B\in \Sigma_{r(A)}$. Then we show that $E_{r}(B)=B$. Indeed, $E_{r}(B)\in \mathbb{R}_{r}$, since $B\in \Sigma_{r(A)}\subseteq \mathbb{R}_{r}$ and $\mathbb{R}_{r} $ is closed under taking $r$-injective hull, by the  former part.   So $E_{r}(B) \leq A$ since $A$ is an $r$-injective $S$-act, containing $B$. Thus there exists $C\in \Sigma_{r(A)}$ such that $E_{r}(B)\leq C$ since $\nabla_{E_{r}(A)}=r(E_{r}(B))\leq r(A)_{\upharpoonright_{E_{r}(B)}}$ and $r$ is a  Kurosh-Amitsur radical. Now $B=E_{r}(B)=C$ follows from this fact that the subacts in $\Sigma_{r(A)}$ are disjoint and $B\leq E_{r}(A)\leq C$.
\end{proof}

In the sequel we are going to define a Kurosh-Amitsur radical $r_{G}$ such that the injective $S$-acts, with respect to $r_{G}$-monomorphisms are exactly the injective $S$-acts.
\medskip

For every  $S$-act $A$ and a zero element $\theta $ of $A$, we  define
\[X_{\theta}:=\bigcup\{C_{\theta}\ |\ C_{\theta}\text{ is a cyclic subact of } A \text{ such that  } \forall c\in C_{\theta} \ \exists s\in S, sc= \theta \} \]
 and  $\mathcal{Z}_{A}=\{\theta\ |\ \theta  \text{ is a zero element of } A\}$. We claim that the following assignment is a Kurosh-Amitsur radical.
\[ r_{G}:A\mapsto r_{G}(A)=\left\{ \begin{matrix}
\displaystyle{\bigvee_{\theta\in Z_{A}}} \rho_{X_{\theta}}   & Z(A)\neq \emptyset &\\
 \Delta_{A} \ & otherwise &
\end{matrix}\right.\]
Indeed, for
 \[\mathbb{S}_{r_{G}}=\{A\ | \text{  every non-trivial  subact of } A \text{ has a cyclic subact without zero}\} \]
 and
 \[\mathbb{R}_{r_{G}}=\{A\ |\  A \text{ has a zero element  $\theta_{A}$ such that} \ \forall a\in A\ \exists s\in S,\ sa=\theta\},\]
  we have
\begin{enumerate}
\item  $\mathbb{R}_{r_{G}}\cap \mathbb{S}_{r_{G}}$ consists of trivial $S$-acts,
\item   $\mathbb{R}_{r_{G}}$ is homomorphically closed,
\item   $\mathbb{S}_{r_{G}}$ is closed under taking subacts,
\item  every $S$-act $A$ has  $\mathbb{R}_{r_{G}}$-$system$ $\Sigma=\{ X_{\theta}\}_{\theta\in Z_{A}}$ whose Rees factor, $A/\rho_{_{\Sigma}}$, belong to $\mathbb{S}_{r_{G}}$.
\end{enumerate}
Therefore $r_{G}$ is a Kurosh-Amitsur radical, by Lemma \ref{pair kurosh-amitsur}.

 It worth noting that since the radical class $\mathbb{R}_{r_{G}}$ is closed under taking subacts,   $r_{G}$ is hereditary, by Theorem \ref{6k.h.}.

\begin{theorem}\label{r-inj. inj.}
The $r_{G}$-injective $S$-acts are exactly the injective $S$-acts.
\end{theorem}

\begin{proof}
Let $I$ be an $r_{G}$-injective $S$-act. Then, using Skornjakov criterion, we show that $I$ is injective with respect to the cyclic subacts. So consider the following diagram in which $B$ is a cyclic subact of $A$.
\[
\xymatrix{
& B  \ar@{^{(}->}[r]^(.35){} \ar[d]_{f}& A  \\
& I  \\
 }\]
 Then  there exists a maximal congruence $\kappa$ with $\kappa\wedge \rho_{B}=\Delta$. By Theorem \ref{essential congruence on A/kappa} and Lemma \ref{rho vee kappa/kappa},  the image of $B$ under canonical epimorphism  $\pi:A\to A/\kappa$ is large in   $A/\kappa$. Hence, for every $[a]_{\kappa}\in A/\kappa$, there is $s\in S$ such that $[sa]_{\kappa}\in A/\kappa$, by Lemma \ref{cap-large}. Thus from the definition of $\mathbb{R}_{r_{G}}$  we have $A/\kappa/ \pi(B)\in \mathbb{R}_{r_{G}}$. Now  since, by Lemma \ref{essential subact},  $B$ is isomorph whit $\pi(B)$, there exists an extension $\widehat{f}:A/\kappa\to  I$ of $f$. Therefore the map $\overline{f}: A\to I$ with $\overline{f}(a)=\widehat{f}([a]_{\kappa})$ is an extension of $f$. So $I$ is injective.
\end{proof}

With Corollary \ref{C-Bear} and  Theorem \ref{r-inj. inj.}  in mind we give an stronger version of Bear-Skornjakov criterion for injectivy,  see the following corollary.

\begin{corollary}\label{cor. inj.}
An $S$-act  $ I $ is injective if and only if $I$ has a zero element and each homomorphism $ f:B\longrightarrow I $ in which $B$ is a large subact of a cyclic $S$-act  $A$  can be extended to $A$.
\end{corollary}

\end{document}